\documentclass[12pt]{amsart}
\usepackage{amssymb, amsmath}
\usepackage{stmaryrd}

\usepackage[dvipsnames]{xcolor}
\usepackage{color}
\usepackage[utf8]{inputenc}
\usepackage[T1]{fontenc}
\usepackage{fullpage}
\usepackage[normalem]{ulem}
\usepackage{amsmath,amscd,mathrsfs,enumerate}
\usepackage[colorlinks=true, citecolor=blue, linkcolor=red,pagebackref, hyperindex]{hyperref}
\usepackage[all,cmtip]{xy}

\newtheorem{Theorem}{Theorem}[section]
\newtheorem{Theoremx}{Theorem}

\newtheorem{Potential Theorem}[Theorem]{Potential Theorem}
\newtheorem{Lemma}[Theorem]{Lemma}
\newtheorem{Corollary}[Theorem]{Corollary}
\newtheorem{Proposition}[Theorem]{Proposition}
\newtheorem{Claim}[Theorem]{Claim}
\newtheorem*{Claim*}{Claim}

\theoremstyle{definition}
\newtheorem{Example}[Theorem]{Example}

\newtheorem{Definition}[Theorem]{Definition}
\newtheorem{Question}[Theorem]{Question}

\theoremstyle{remark}
\newtheorem{Remark}[Theorem]{Remark}

\DeclareMathOperator{\height}{ht}

\DeclareMathOperator{\id}{id}
\DeclareMathOperator{\Proj}{Proj}

\DeclareMathOperator{\Hom}{Hom}
\DeclareMathOperator{\Spec}{Spec}

\DeclareMathOperator{\Ass}{Ass}

\DeclareMathOperator{\Ext}{Ext}

\DeclareMathOperator{\coker}{coker}

\DeclareMathOperator{\rsig}{s_{rat}}
\DeclareMathOperator{\csig}{s_{rel}}
\DeclareMathOperator{\ehk}{e_{HK}}
\DeclareMathOperator{\length}{\lambda}

\DeclareMathOperator{\s}{s}

\DeclareMathOperator{\sdim}{sdim}

\newcommand{\Tr}{\operatorname{Tr}}
\newcommand{\ds}{\displaystyle}

\def\p{\mathfrak{p}}

\def\m{\mathfrak{m}}
\def\n{\mathfrak{n}}
\def\a{\mathfrak{a}}

\def\Z{\mathbb{Z}}
\def\R{\mathbb{R}}
\def\Q{\mathbb{Q}}

\def\F{\mathbb{F}}

\def\N{\mathbb{N}}

\def\CC{\mathbb{C}}

\renewcommand{\geq}{\geqslant}
\renewcommand{\leq}{\leqslant}

\newcommand{\ck}[1]{{#1}^{\vee}}
\DeclareMathOperator{\ann}{ann}
\newcommand{\ps}[1]{\llbracket {#1} \rrbracket}

\newcommand{\Att}{\operatorname{Att}}

\newcommand{\soc}{\operatorname{soc}}

\newcommand{\pow}[2]{(#1)^{\oplus #2}}

\newcommand{\ul}{\underline}
\newcommand{\ov}{\overline}
\DeclareMathOperator{\cP}{\mathcal{P}}

\usepackage{comment}

\begin{document}

\title{Stability and deformation of F-singularites}
\author{Alessandro De Stefani}
\address{Dipartimento di Matematica, Universit{\`a} di Genova, Via Dodecaneso 35, 16146 Genova, Italy}
\email{destefani@dima.unige.it}

\author{Ilya Smirnov}
\address{Department of Mathematics, Stockholm University, S-10691, Stockholm, Sweden}
\email{smirnov@math.su.se}

\subjclass[2010]{13A35, 13H10, 14B05, Secondary: 13D45,  14B07}

\begin{abstract}
We study the problem of $\m$-adic stability of F-singularities, that is, whether the property that a quotient of a local ring $(R,\m)$ by a non-zero divisor $x \in \m$ has good F-singularities is preserved in a sufficiently small $\m$-adic neighborhood of $x$. We show that $\m$-adic stability holds for F-rationality in full generality, and for F-injectivity, F-purity and strong F-regularity under certain assumptions. We show that strong F-regularity and F-purity are not stable in general. Moreover, we exhibit strong connections between stability and deformation phenomena, which hold in great generality.
\end{abstract}

\maketitle

\section{Introduction}

When dealing with solutions of a system of equations, one often desires to have solutions of a nice class, such as algebraic or, even, polynomial functions. Thus, very important are results that allow us to approximate an existing solution with solutions in a desired class.
For example, the celebrated Artin's approximation allows us to approximate a formal power series solution with convergent power series: 
given a system of polynomial equations in two sets of variables $\underline{x}$ and $\underline{y}$ and a formal solution $\hat{\underline{y}}$,
then for any $c > 0$ there will be an algebraic power series solution $\underline{y}(\underline{x})$ 
such that $\underline{y}(\underline{x}) \equiv \hat{\underline{y}}(\underline{x}) \mod (\underline{x})^c$. 

In a related way, Samuel's problem of \emph{finite determinacy} asks to truncate the formal powers series $\hat{\underline{y}}$ to obtain 
an equivalent algebraic singularity. In \cite{Samuel} Samuel proved this for a hypersurface isolated singularity in a much stronger form by showing that for $c$ large enough any approximation in the above sense defines an equivalent singularity. 
Unfortunately, if the singularity is not isolated, this is no longer true (\cite{GreuelPham}),
and it is natural to investigate what properties can be preserved.

Specifically, we say that a ring property $\cP$ is \emph{stable under fine perturbations}
(or \emph{$\m$-adically stable}) if, whenever a local ring $(R,\m)$ and a non-zero divisor $x \in \m$ are such that $R/(x)$ satisfies $\cP$, then $R/(y)$ satisfies $\cP$ for all $y$ in a sufficiently small $\m$-adic neighborhood $x+\m^N$ of $x$.
While $\m$-adic stability is clearly a stronger property than finite determinacy, 
it is also easier to work with in practice. In fact, we are not aware of another common way to establish finite determinacy.

In this article, we study the behavior of F-singularities in this context, a direction that originates from \cite{PolstraSmirnov, MaPhamSmirnov}. F-singularity is an umbrella term for the classes of singularities defined by the properties of the Frobenius endomorphism. They originate from the theory of tight closure (\cite{HochsterHuneke}) and homological conjectures in positive characteristic (\cite{HochsterRoberts}),  but are also connected with intrinsic classes of singularities in birational geometry and, after recent work, with singularities in mixed characteristic. 
The four most fundamental classes of F-singularities are strongly F-regular, F-pure (Definition~\ref{def F-sing}), F-rational (Definition~\ref{def F-rat}), and F-injective (Definition~\ref{def F-inj}) singularities, related in the following way:
\[
\xymatrixcolsep{1.5cm}
\xymatrix@1{
\text{strong F-regular} \ar@{=>}[dd] \ar@{=>}[r] & \text{F-rational} \ar@{=>}[dd] \\ \\
\text{F-pure} \ar@{=>}[r] & \text{ F-injective}
}
\]

In Section \ref{Section perturbation F-rat} we prove stability of F-rationality. In fact, we prove more: Hochster and Yao defined in \cite{HochsterYao}  a numerical invariant, called the F-rational signature $\rsig(R)$, which has the property that $\rsig(R)>0$ if and only if $R$ is F-rational (see Definition \ref{Definition F-rational signature} and the subsequent discussion). 
Our main result in Section \ref{Section perturbation F-rat} is the following:
\begin{Theoremx}[see Theorem \ref{continuity of rational signature} and Corollary \ref{F-rational is ok}] \label{THMX F-rational}
Let $(R, \mathfrak m)$ be an F-finite local ring of prime characteristic $p>0$. If $R$ is Cohen-Macaulay, and $x \in \m$ is a non-zero divisor such that $\widehat{R}/(x)\widehat{R}$ is reduced, then for any $\varepsilon > 0$ there exists an integer $N > 0$ such that
\[
|\rsig (R/(x))- \rsig (R/(x+\delta)) | < \varepsilon
\]
for any $\delta \in \m^N$. In particular, if $R/(x)$ is F-rational, then $R/(x+\delta)$ is F-rational.
\end{Theoremx}
For Gorenstein rings the F-rational signature coincides with the F-signature. Thus Theorem \ref{THMX F-rational} extends and gives a different perspective on \cite[Theorem~3.12]{PolstraSmirnov}. 

One can see that stability is a natural $\m$-adic analogue of \emph{deformation}. 
Let $X$ be an algebraic variety over a  field $K$ and consider a one-parameter deformation of $X$, that is,
a flat morphism $f\colon Y \to \Spec(K[t])$ such that $X \cong Y_0 := f^{-1}(0)$. 
We say that a property $\cP$ deforms if, whenever $X$ satisfies $\cP$, then 
$Y_r = f^{-1}(r)$ satisfies $\cP$ for $r$ in a neighborhood of the special fiber $Y_0$.
For $K = \CC$ this just means that $r$ is taken sufficient close to $0$, see \cite[Chapter 9]{Ishii} for results in this direction. 

A common approach for establishing deformation is by proving first that the total space $Y$ necessarily satisfies $\cP$ in a neighborhood of $0$
and then descending $\cP$ to $Y_r$. In the local case, the former requires that if  $R/(x)$ satisfying $\cP$ for some non-zero divisor $x$ then $R$ satisfies $\cP$. We will follow recent literature and, from now on, will say in such case that a property of local rings $\cP$ \emph{deforms}. 

The study of deformation of F-singularities by Fedder in 1987: motivated by an analogous result of Elkik for rational singularities \cite{Elkik}, Fedder (\cite{Fedder}) proved that F-purity does not deform in general, but it does when $R$ is Gorenstein. He also showed that F-injectivity deforms when $R$ is Cohen-Macaulay. Hochster and Huneke proved that F-rationality deforms \cite{HochsterHunekeSmoothBaseChange}. F-regularity deforms in some special cases \cite{AKM, SinghFPURFREG}, but, in general, Singh provided a counterexample  \cite{SinghNODeformationFREG}. There have been several recent results in the direction of proving deformation of F-injectivity \cite{HMS,MaSchwedeShimomoto,MaQuy}, but this remains an open problem.

Our next main result is a general relation between deformation and stability. 
\begin{Theoremx}[see Theorem \ref{THM P->D}] \label{THM A}
Let $\cP$ be a property of local rings that satisfies the following two conditions:
\begin{itemize}
\item[a)] If $(R,\m)$ satisfies $\cP$ and $T$ is a variable, then $R[T]_{(\m,T)}$ satisfies $\cP$.
\item[b)] If $(R,\m) \to (S,\n)$ is faithfully flat, and $S$ satisfies $\cP$, then $R$ satisfies $\cP$.
\end{itemize}
If $\m$-adic stability of $\cP$ holds, then deformation of $\cP$ holds.
\end{Theoremx}

The aforementioned four classes of F-singularities 
satisfy the two conditions of Theorem \ref{THM A} (see \cite{Hashimoto, Velez, HochsterRoberts, DattaMurayama}), thus stability of such F-singularities implies deformation. 

In Section \ref{Section relation F-injective and perturbation}, we consider the problem of $\m$-adic stability of F-injectivity. While we are not able to prove this in full generality, we show that F-injectivity is stable in the most relevant cases where deformation is known to hold.
\begin{Theoremx}[see Theorems \ref{PerturbationFInjective SurjectiveElement} and \ref{Theorem F-inj strictly filter}] Let $(R,\m)$ be a local ring of prime characteristic $p>0$, and $x \in \m$ be a non-zero divisor such that $R/(x)$ is F-injective.
\begin{enumerate}
\item If $x$ is a surjective element, then there exists an integer $N >0$ such that $R/(x+\delta)$ is F-injective for all $\delta \in \m^N$.
\item If $R/\m$ is perfect and $x$ is a strictly filter regular element, then there exists an integer $N >0$ such that $R/(x+\delta)$ is F-injective for all $\delta \in \m^N$.
\end{enumerate}
\end{Theoremx}

Because F-purity and strong F-regularity do not deform \cite{Fedder, SinghNODeformationFREG}, we obtain right away that stability does not hold for such F-singularities. We do provide an explicit example (see Theorem \ref{Counterexample FReg}), based on the construction of Singh \cite{SinghNODeformationFREG}. 

Singh's example also shows that strong F-regularity is not open in families (Example~\ref{Example Family})
and we push this further to show that it may not be open even in a proper flat family (Example~\ref{Proper Family}).
This shows that strong F-regularity has a very different geometric behavior from 
F-rationality, which is known to be open \cite{Hashimoto01}, \cite{PSZ}
even in mixed characteristic \cite[Theorem~7.2]{MaSchwede}.

It is known that deformation of strong F-regularity holds under some additional assumptions; for instance, when $(R,\m)$ is $\Q$-Gorenstein (\cite{AKM} and \cite{MacCrimmonThesis}, or \cite{AberbachWeakAndStrong}). We obtain an analogous statement for stability.
\begin{Theoremx}[see Theorem \ref{Theorem Q-Gorenstein}] \label{THMX Q-Gorenstein}
Let $(R, \m)$ be an F-finite local ring of characteristic $p > 0$ that is $\Q$-Gorenstein on the punctured spectrum of $R$. If $x$ is a non-zero divisor such that $R/(x)$ is strongly F-regular, then for any $\varepsilon > 0$ there exists an integer $N > 0$
such that
\[
\left| \s(R/(x)) - \s(R/(x+\delta)) \right| < \varepsilon
\]
for all $\delta \in \m^N$. In particular, $R/(x+\delta)$ is strongly F-regular.
\end{Theoremx}

We also obtain some positive results for stability of F-purity under some additional assumptions (see subsection \ref{Section perturbation F-pure}). More generally, we prove that sharp F-purity, introduced by Schwede in \cite{SchwedeSharpFPurity}, is $\m$-adically stable (see Theorem \ref{Perturbation FPurity Compatible} and Proposition \ref{PropositionSharpFPurity}).

Our results show a striking difference between F-signature and Hilbert--Kunz multiplicity. The
Hilbert--Kunz multiplicity is open in families (see \cite{SmirnovFamily}), and 
is continuous in the $\m$-adic topology by \cite[Corollary 3.7]{PolstraSmirnov}. 
Theorem \ref{Counterexample FReg} is a counter-example to \cite[Corollary~3.10]{PolstraSmirnov}
and \cite[Theorem~3.12]{PolstraSmirnov}, which by Theorem \ref{THMX Q-Gorenstein} are only known to be true for rings which are $\Q$-Gorenstein on the punctured spectrum\footnote{A valid proof of \cite[Theorem~3.12]{PolstraSmirnov} for Gorenstein rings will appear as an erratum, and is currently available on the arXiv.}.


\subsection*{Acknowledgments} The authors would like to thank Linquan Ma, Pham Hung Quy, Karl Schwede, and Anurag Singh for useful discussions and for sharing insights and ideas on the topics of this paper, Thomas Polstra for pointing out an inaccuracy in an earlier version, and Austyn Simpson for comments.  

\section{General relations between stability and deformation}
Throughout, $(R,\m)$ will denote a Noetherian local ring.

\begin{Definition} Let $\cP$ be a given property of a local ring. We say that {\it deformation of $\cP$} holds if, whenever $(R,\m)$ is a local ring and $x \in \m$ is a non-zero divisor such that $R/(x)$ satisfies $\cP$, then $R$ satisfies $\cP$.
\end{Definition}

We now formally introduce the notion of $\m$-adic stability for a property of a local ring.
\begin{Definition}
Let $\cP$ be a given property of a local ring. 
We say that {\it $\m$-adic stability of $\cP$} holds if, whenever $(R,\m)$ is a local ring and $x \in \m$ is a non-zero divisor such that $R/(x)$ satisfies $\cP$, there exists an integer $N$ such that $R/(x+\delta)$ satisfies $\cP$ whenever $\delta \in \m^N$.
\end{Definition}

\begin{Remark}
Both definitions of deformation and $\m$-adic stability naturally generalize to regular sequences. We point out that, while deformation for regular elements is clearly equivalent to deformation for regular sequences, a priori this is not the case for stability.
\end{Remark}

The introduction of $\m$-adic stability is motivated by recent results of Ma, Pham and Smirnov \cite{MaPhamSmirnov}, who prove that small perturbation of quotients by filter regular sequences have isomorphic associated graded rings. The study of $\m$-adic continuity of two important invariants in positive characteristic, the Hilbert--Kunz multiplicity and the F-signature, was also initiated in \cite{PolstraSmirnov}, and will be further developed in the next sections of this article.

Our main goal is to understand how F-singularities behave under small perturbations, and to relate $\m$-adic stability to deformation problems. We can find a surprising connection between these two notions, that holds in great generality.

\begin{Theorem} \label{THM P->D}
Let $\cP$ be a property of local rings such that
\begin{enumerate}
\item If a local ring $R$ satisfies $\cP$, and $T$ is a variable, then $R[T]_{(\m,T)}$ satisfies $\cP$.
\item If $(R,\m) \to (S,\n)$ is faithfully flat, and $S$ satisfies $\cP$, then so does $R$.
\end{enumerate} 
If $\m$-adic stability holds for $\cP$, then deformation holds for $\cP$.
\end{Theorem}
\begin{proof}
Assume that deformation does not hold for $\cP$. This means that there exists a local ring $(R,\m)$ and a non-zero divisor $x \in \m$ such that $R/(x)$ satisfies $\cP$, but $R$ does not. For $N \geq 1$, consider the rings $A_N=R[T^N]_{(\m,T^N)}$ and $B=R[T]_{(\m,T)}$, where $T$ is a variable on $R$. Clearly, $x$ is still a non-zero divisor on $B$. It is easy to see that $B$ is free over $A_N$, and it follows that $R_N=B/(x-T^N)B$ is free over $S_N=A_N/(x-T^N)A_N$. In particular, $S_N \to R_N$ is faithfully flat. However, observe that $S_N \cong R$ does not satisfy $\cP$, hence $R_N$ does not satisfy $\cP$ thanks to condition (2). Finally, observe that $B/(x) \cong (R/(x))[T]_{(\m,T)}$. Since $R/(x)$ satisfies $\cP$, so does $B/(x)$ by condition (1).

We have shown that, for the local ring $B$, the non-zero divisor $x$ is such that $B/(x)$ satisfies $\cP$, but $B/(x-T^N)$ does not satisfy $\cP$ for any $N \geq 1$. Hence $\cP$ is not stable.
\end{proof}

\begin{Remark} \label{Remark strictly filter} In the notation of the proof of Theorem \ref{THM P->D}, one can show that $x-T^N$ is actually a strictly filter regular element in $B$ for all $N \geq 1$ (see Definition \ref{Definition strictly filter regular}). So the same proof actually shows that, if $\m$-adic stability of $\cP$ holds for all strictly filter regular elements, then deformation of $\cP$ holds.
\end{Remark}

\section{Stability of F-rationality} \label{Section perturbation F-rat}

The goal of this section is to prove that $\m$-adic stability of F-rationality holds in full generality. In light of Theorem \ref{THM P->D}, this gives another proof of the fact that F-rationality deforms \cite{HochsterHunekeSmoothBaseChange}. Our proof will employ the notion of F-rational signature, introduced by Hochster and Yao \cite{HochsterYao}. We start by recalling the relevant definitions and known facts.

Let $(R,\m)$ be a local ring of prime characteristic $p>0$. The Frobenius endomorphism  on $R$ induces maps $F_i\colon H^i_\m(R) \to H^i_\m(R)$ on local cohomology modules supported at the maximal ideal. If no confusion may arise, we will denote $F_i$ just by $F$.

\begin{Definition}\label{def F-rat}
Let $(R,\m)$ be a $d$-dimensional Cohen-Macaulay local ring of prime characteristic. Then $R$ is called {\it F-rational} if, for all elements $c$ of $R$ not in any minimal prime, the Frobenius actions $cF^e \colon H^d_\m(R) \to H^d_\m(R)$ are injective for all integers $e \gg 0$.
\end{Definition}

The condition that $cF^e$ is injective for all $e \gg 0$ and all $c$ not in any minimal prime of $R$ is equivalent to saying that the zero submodule of $H^d_\m(R)$ is tightly closed. The definition of F-rationality is originally formulated in terms of tight closure of parameter ideals, but it is equivalent to Definition \ref{def F-rat} for a Cohen-Macaulay local ring (see \cite{Smith}).

\begin{Lemma} \label{Lemma nzd} Let $(R,\m)$ be a local ring of Krull dimension $d>0$ and $x \in \m$ be a parameter. If $R/(x)$ is a domain, then $R$ is a domain.
\end{Lemma}
\begin{proof}
First we prove that $x$ is a non-zero divisor on $R$, or, equivalently, that  $x$ is not contained in any associated prime of $R$.
We start with minimal primes: assume by way of contradiction that $x \in Q$ for some prime $Q \in \min\Ass(R)$. Clearly, $Q/(x) \in \min\Ass_{R/(x)}(R/(x))$. Now let $Q' \in \min\Ass_R(R)$ be such that $\dim(R/Q') = d$. Because $x$ is a parameter, we must have $x \notin Q'$. Choose a prime $P \in \min\Ass_R(R/(x,Q'))$. Observe that $\dim(R/P)=d-1$, and thus $P/(x) \in \min\Ass_{R/(x)}(R/(x))$. Since $Q$ is a minimal prime of $R$, we have that $Q' \not\subseteq Q$, and thus $P \ne Q$. It follows that $P/(x) \ne Q/(x)$ are two distinct minimal primes of $R/(x)$, which contradicts the assumption that $R/(x)$ is a domain.

Assume now that $x \in Q$ for some associated prime $Q$ of $R$. After localizing at $Q$, we still have that $(R/(x))_Q$ is a domain; hence we may assume without loss of generality that $\m \in \Ass(R)$. Let $y \in \sqrt{0}$. Then $y^n = 0 \in (x)$ for $n \gg 0$. Since $R/(x)$ is a domain, $(x)$ is a prime ideal. It follows that $y \in (x)$, so we can find $y_1 \in R$ such that $y=xy_1$. We then have $y=xy_1 \in \sqrt{0} = \bigcap_{Q \in \min\Ass(R)} Q$.
By the claim, we conclude that $y_1 \in \sqrt{0}$. Repeating this process, we see that for every $n \in \N$ there exists $y_n \in R$ such that $y=x^n y_n$, so that $y \in \bigcap_n (x)^n = (0)$. This shows that $\sqrt{0} = 0$, hence $R$ is reduced. However, this contradicts the assumption that $\m \in \Ass(R)$, and completes the proof that $x$ is a non-zero divisor on $R$. 

To show that $R$ is a domain, assume that $yz = 0$ for some $y,z \in R$. We can write $y=x^{m}y'$ and $z=x^nz'$ for some $y',z' \in R \smallsetminus (x)$. Since we have already proved that $x$ is a non-zero divisor, from $yz = x^{m+n}y'z' = 0$ we conclude that $y'z'=0$. However, this implies that $y'z' \in (x)$, and thus either $y' \in (x)$ or $z' \in (x)$. A contradiction.
\end{proof}

In what follows, we will denote by $\ul{x}=x_1,\ldots,x_c$ a sequence of elements of $R$. Given two such sequences $\ul{x}=x_1,\ldots,x_c$ and $\ul{y}=y_1,\ldots,y_c$ of the same length, we will denote by $\ul{x}+\ul{y}$ the sequence $x_1+y_1,\ldots,x_c+y_c$. We will say that $\ul{x}=x_1,\ldots,x_c$ is a system of parameters of length $c$ if each $x_i$ is a parameter in $R/(x_1,\ldots,x_{i-1})$. We will call $\ul{x}$ a full system of parameters if, in addition, $(\ul{x})$ is an $\m$-primary ideal. 

\begin{Corollary} \label{Corollary F-regular and F-rational nzd} Let $(R,\m)$ be a local ring, and $c$ be an integer such that $\dim(R)>c>0$. Let $\ul{x}$ be a system of parameters of length $c$. If $R/(\ul{x})$ is F-rational, then $R$ is F-rational. In particular, $R$ is a Cohen-Macaulay domain.
\end{Corollary}
\begin{proof}
It suffices to show the statement for $c=1$, and $\ul{x}=x$ a parameter. Since $R/(x)$ is F-rational, it is a domain. It follows from Lemma \ref{Lemma nzd} that $R$ is a domain, hence $x$ is a non-zero divisor in $R$. Since F-rationality deforms \cite{HochsterHunekeSmoothBaseChange}, it follows that $R$ is F-rational.
\end{proof}

\begin{Remark}
Linquan Ma has pointed out to us a different proof of Corollary \ref{Corollary F-regular and F-rational nzd}: after completing $\underline{x}$ to a full system of parameters $\underline{x},\underline{y}$ on $R$, we observe that the tight closure of the ideal $(\underline{x},\underline{y})$ in $R$ is contained in the lift to $R$ of the tight closure of the ideal $(\underline{y})$ in $R/(\underline{x})$, by persistence of tight closure. As the latter is trivial by assumption, it follows that $(\underline{x},\underline{y})$ is tightly closed, and hence $R$ is F-rational.
\end{Remark}

\begin{Definition}[\cite{HochsterYao} and \cite{SmirnovTucker}] \label{Definition F-rational signature}
Let $(R, \m)$ be a Cohen-Macaulay local ring and $\ul{x}$ be a full system of parameters in $R$. The F-rational signature of $R$ is defined as 
\[
\rsig(R) = \inf \left \{ \ehk(\ul{x}) - \ehk(\ul{x}, u) \mid u \notin (\ul{x}), \m u \subseteq (\ul{x}) \right \}.
\]

A closely related invariant, the relative F-rational signature of $R$, is defined as 
\[
\csig(R) = \inf \left \{ \frac{\ehk(\ul{x}) - \ehk(I)}{\length (R/(\ul{x})) - \length (R/I)} \ \bigg| \ (\ul{x}) \subsetneq I, \m I \subseteq (\ul{x}) \right \}.
\]
\end{Definition}
It is known that these invariants are independent on the choice of $\ul{x}$.
The main property of the F-rational signature is that, for a Cohen-Macaulay local ring $R$, one has $\rsig(R) > 0$ if and only if $R$ is F-rational. 
It is not hard to see that $\length ((\ul{x}:\m)/(\ul{x}))\csig(R) \geq \rsig(R) \geq \csig(R)$, and in particular 
$R$ is F-rational if and only $\csig(R)>0$ \cite[Corollary 4.12]{SmirnovTucker}.

Our goal is to show that these invariants are $\m$-adically continuous as it was shown for Hilbert--Kunz multiplicity in \cite{PolstraSmirnov}. We start a modification of a uniform convergence result for Hilbert-Kunz multiplicities that was proved By Polstra and the second author in \cite[Corollary~3.6]{PolstraSmirnov}.

\begin{Theorem}\label{very uniform convergence}
Let $(R, \mathfrak m)$ be a $d$-dimensional F-finite Cohen-Macaulay local ring of prime characteristic $p>0$, and $\ul{x}$ be a system of parameters of length $c>0$ such that $\widehat{R}/(\ul{x})\widehat{R}$ is reduced.
For any integer $e_0 > 0$ there exist a constant $C \geq 0$ and an integer $N\geq 1$ such that, for any $e\geq 1$, any $\ul{\delta} \in \pow{\mathfrak m^N} c$, and any ideal $I$ such that $\m^{[p^{e_0}]} \subseteq I$, one has
\[
\left |\ehk (I, R/(\ul{x}+\ul{\delta})) - p^{-ed} \length \left (R/(\ul{x} + \ul{\delta}, I^{[p^e]}) \right) \right| \leq C p^{-e}.
\]
\end{Theorem}
\begin{proof}
For $I = \m$ the statement was proved in \cite[Corollary~3.6]{PolstraSmirnov}, but its proof and 
the proof of \cite[Theorem~3.5]{PolstraSmirnov} applies to any ideal. 
To see the uniformity one can further follow the proofs of \cite{Tucker} and note that the right-hand side constant $C(I)$
for an ideal $I$ comes from a fixed module, and, thus, will be bounded above by the constant for $C(\m^{[p^{e_0}]}) = p^{e_0} C(\m)$. See also \cite[Theorem~4.3]{PolstraTucker}.
\end{proof}

\begin{Corollary}\label{ehk uniformly close}
Let $(R, \m)$ be a $d$-dimensional F-finite Cohen-Macaulay local ring of prime characteristic $p>0$, and $\ul{x}$ be a system of parameters of length $c>0$ such that $\widehat{R}/(\ul{x})\widehat{R}$ is reduced.
For any $\varepsilon > 0$ and any integer $e_0 > 0$, there exists an integer $N\geq 1$ such that for any 
$\ul{\delta} \in \pow{\mathfrak m^N} c$ and any ideal $I$ such that $\m^{[p^{e_0}]} \subseteq I$, one has 
\[
\length (R/(I, \ul{x})) = \length (R/(I, \ul{x} + \ul{\delta})) \text { and }
\left |\ehk (I, R/(\ul{x}+\ul{\delta})) - \ehk (I, R/(\ul{x}))  \right| < \varepsilon
\]
\end{Corollary}
\begin{proof}
We apply Theorem~\ref{very uniform convergence} to get $C$ and we choose $e$ so that $Cp^{-e} < \varepsilon/2$. 
We may further enlarge $N$ so that $\m^N \subseteq (\m^{[p^{ee_0}]},\ul{x})$, so that 
for all $\m^{[p^{e_0}]} \subseteq I$ we have
\[
\length \left (R/(\ul{x} + \ul{\delta}, I^{[p^e]}) \right) = \length \left (R/(\ul{x}, I^{[p^e]}) \right).
\]
Then by Theorem~\ref{very uniform convergence} 
\begin{align*}
|\ehk(I, R/(\ul{x})) - \ehk(I, R/(\ul{x} +  \ul{\delta}))| &\leq 
\left |\ehk (I, R/(\ul{x}+\ul{\delta})) - p^{-ed} \length \left (R/(\ul{x} + \ul{\delta}, I^{[p^e]}) \right) \right|
\\ &+ \left |\ehk (I, R/(\ul{x})) - p^{-ed} \length \left (R/(\ul{x}, I^{[p^e]}) \right) \right| < \varepsilon. \qedhere
\end{align*} 
\end{proof}

\begin{Theorem} \label{continuity of rational signature}
Let $(R, \m)$ be a $d$-dimensional F-finite Cohen-Macaulay local ring of prime characteristic $p>0$, and $\ul{x}$ be a system of parameters of length $c>0$ such that $\widehat{R}/(\ul{x})\widehat{R}$ is reduced. For any $\varepsilon > 0$ there exists an integer $N \geq 1$ such that for any $\ul{\delta} \in \pow{\mathfrak m^N} c$
\[
|\rsig (R/(\ul{x})) - \rsig (R/(\ul{x} + \ul{\delta})) | < \varepsilon
\]
and
\[
|\csig (R/(\ul{x})) - \csig (R/(\ul{x} + \ul{\delta})) | < \varepsilon.
\]
\end{Theorem}
\begin{proof}
Fix $\varepsilon > 0$. If $\csig (R/(\ul{x})) = 0$, then $\rsig(R/(\ul{x})) = 0$, so there exists a system of parameters $\ul{y}$ in $R/(\ul{x})$ and a non-zero element $u \in \soc(R/(\ul{x},\ul{y}))$ such that $\ehk((\ul{y}, u), R/(\ul{x})) = \ehk((\ul{y}), R/(\ul{x})) = \length (R/(\ul{x}, \ul{y}))$. Take $e_0$ such that $\m^{[p^{e_0}]} \subseteq (\ul{x}, \ul{y})$ and apply Corollary~\ref{ehk uniformly close} to get that 
\[
\csig (R/(\ul{x} + \ul{\delta})) \leq \rsig (R/(\ul{x} + \ul{\delta})) \leq \length (R/(\ul{x} +  \ul{\delta},\ul{y})) - \ehk((\ul{y}, u), R/(\ul{x}+ \ul{\delta})) < 
\varepsilon.
\]

Now assume $\csig(R) > 0$. Applying Corollary~\ref{ehk uniformly close} we can find $N$ such that for all $\ul{\delta} \in \pow{\mathfrak m^N} c$ and any ideal $(\ul{x}, \ul{y}) \subseteq I$
\[
\ehk (I, R/(\ul{x})) - \varepsilon <  \ehk (I, R/(\ul{x}+\ul{\delta})) < \ehk (I, R/(\ul{x})) + \varepsilon.
\]
Moreover, because $\length(R/(\ul{x},\ul{y})) = \length(R/(\ul{y},\ul{x}+\ul{\delta}))$ and $ \length(R/(I,\ul{x})) = \length(R/(I,\ul{x}+\ul{\delta}))$, it easily follows that
\begin{align*}
\frac{ \length (R/(\ul{x}, \ul{y})) - \ehk(I, R/(\ul{x}))}{ \length (R/(\ul{x}, \ul{y})) - \length (R/(I, \ul{x}))}
- \varepsilon
& <
\frac{ \length (R/(\ul{x} +  \ul{\delta},\ul{y})) - \ehk(I, R/(\ul{x}+ \ul{\delta}))}{ \length (R/(\ul{x} +  \ul{\delta},\ul{y})) - \length (R/(I, \ul{x} + \ul{\delta}))} \\
& 
< \frac{ \length (R/(\ul{x}, \ul{y})) - \ehk(I, R/(\ul{x}))}{ \length (R/(\ul{x}, \ul{y})) - \length (R/(I, \ul{x}))}
+ \varepsilon.
\end{align*}
The statement is clear after taking the infimum over all ideals $I$ containing $(\ul{y})$. The proof for the F-rational signature is completely analogous.
\end{proof}

\begin{Corollary}\label{F-rational is ok}
Let $(R,\m)$ be a $d$-dimensional F-finite local ring, and let $\ul{x}$ be a system of parameters of length $c>0$ such that $R/(\ul{x})$ is F-rational. There exists an integer $N > 0$ such that $R/(\ul{x}+\ul{\delta})$ is F-rational for any $\ul{\delta} \in \pow{\m^N} c$.
\end{Corollary}
\begin{proof}
Let $r=\rsig(R/(\ul{x}))$, so that $r>0$ because $R/(\ul{x})$ is assumed to be F-rational. By Lemma \ref{Lemma nzd}, and because $\widehat{R}/(\ul{x}) \widehat{R}$ is F-rational, hence a domain, the assumptions of Theorem \ref{continuity of rational signature} are satisfied. Therefore we can find an integer $N>0$ such that for every $\delta \in (\m^N)^{\oplus c}$ we have
\[
\left| \rsig(R/(\ul{x}+\ul{\delta})) - r \right| < \frac{r}{2}.
\]
In particular, we have $\rsig(R/(\ul{x}+\ul{\delta})) > \frac{r}{2}>0$, and thus $R/ (\ul{x}+\ul{\delta})$ is F-rational.
\end{proof}

\section{Stability of F-injectivity} \label{Section relation F-injective and perturbation}

We start by reviewing the notions of secondary representation and attached prime. We refer to \cite[7.2]{BrodmannSharp} for unexplained notation and for more details.

\begin{Definition} A \emph{secondary representation} of an $R$-module $W$ is an expression of the form $W=W_1 + \cdots + W_t$, where each $W_i$ is such that
\begin{itemize}
\item $\sqrt{\ann_R(W_i)} = \p_i$ is prime, and
\item given $x \in R$ , the multiplication map $W_i \stackrel{\cdot x}{\longrightarrow} W_i$ is either nilpotent or surjective.
\end{itemize}
The representation is called minimal if $\p_1,\ldots,\p_t$ are all distinct primes, and $W_i \not\subseteq \sum_{j \ne i} W_j$ for all $i=1,\ldots,t$.
\end{Definition}
Modules satisfying the two conditions above are called \emph{secondary}. Just like primary decompositions, secondary representations are not unique in general. However, if $W=W_1+\cdots+W_t$, the set of primes $\Att_R(W)=\{\p_1,\ldots,\p_t\}$ defined as above is an invariant of the module $W$, and it is called the set of {\it attached primes of $W$}.

Evey Artinian $R$-module has a secondary representation \cite[7.2.9]{BrodmannSharp}. In particular, $H^i_\m(R)$ has a secondary representation for every $i \in \Z$, and therefore it makes sense to consider the sets $\Att(H^i_\m(R))$.

\begin{Remark} Let $\ck{W} = \Hom_R(W,E_R(R/\m))$ denote the Matlis dual of a module. When $R$ is complete, one can prove that $\Att(W) = \Ass(\ck{W})$. In particular, if $R$ is a quotient of an $n$-dimensional complete regular local ring $S$, then $\Att(H^i_\m(R)) = \Ass(\Ext^{n-i}_S(R,S))$.
\end{Remark}

\begin{Definition} \label{Definition surjective element} Let $(R,\m)$ be a local ring. An element $x \in \m$ is a \emph{surjective element} if the multiplication map $H^i_\m(R) \stackrel{\cdot x}{\longrightarrow} H^i_\m(R)$ is surjective for all $i \in \Z$. 
\end{Definition}

The definition given above is not the original one, given in \cite{HMS}. However, it was proved to be equivalent in \cite[Proposition 3.3]{MaQuy}.

It follows directly from the definition and the properties of secondary representations that  $x \in \m$ is a surjective element if and only if $x \notin \bigcup_{i=0}^d \Att(H^i_\m(R))$. In particular, if $R$ has a surjective element, then $\m$ is not attached to any local cohomology module $H^i_\m(R)$. Moreover, since $\Ass(R) \subseteq \bigcup_{i=0}^{d} \Att(H^i_\m(R))$ \cite[11.3.9]{BrodmannSharp}, surjective elements are non-zero divisors.

\begin{Definition} \label{Definition strictly filter regular} Let $(R,\m)$ be a local ring of dimension $d$. 
An element $x \in \m$ is called  {\it strictly filter regular} if $x \notin \left(\bigcup_{i=0}^{d} \Att(H^i_\m(R)) \smallsetminus  \{\m\}\right)$.
\end{Definition}

Clearly, surjective elements are strictly filter regular. In fact, it is easy to see that $x$ is a strictly filter regular element if and only if $\coker(H^i_\m(R) \stackrel{\cdot x}{\longrightarrow} H^i_\m(R))$ has finite length for all $i \in \Z$.

Strictly filter regular elements are $\m$-adically stable.

\begin{Lemma}\label{Lemma sfr perturbs}
Let $(R,\m)$ be a local ring, and $x$ be a strictly filter regular element.
There exists an integer $N > 0$ such that for any $\delta \in \m^N$ 
the element $x + \delta$ is strictly filter regular. Moreover, if $x$ is a surjective element, 
then $x + \delta$ is also a surjective element. 
\end{Lemma}
\begin{proof}
Let $d=\dim(R)$, and set $X =  \bigcup_{i = 0}^d \Att(H^i_\m(R))$, and $P \in X \smallsetminus \{\m\}$. Since $\bigcap_N (P+\m^N) = P$, we may take $N$ such that $x \notin (\m^N + P)$. Choosing $N$ that works for every such $P$, for all $\delta \in \m^N$ we then have that $x+\delta \notin P$ for all $P \in X \smallsetminus  \{\m\}$. If $x$ is a surjective element then $\m \notin X$. With the choice made above, for all $\delta \in \m^N$ we have $x+ \delta \notin P$ for all $P \in X$, that is, $x$ is a surjective element.
\end{proof}

\begin{Definition} \label{def F-inj}
Let $(R,\m)$ be a local ring of prime characteristic $p>0$. $R$ is called {\it F-injective} if the Frobenius action $F\colon H^i_\m(R) \to H^i_\m(R)$ is injective for all $i \in \Z$.
\end{Definition}

Deformation of F-injectivity is an important open problem that dates back to Fedder \cite{Fedder}. It is known that F-injectivity deforms in many cases; for instance, if $x$ is a surjective element \cite{HMS}, or if $x$ is a strictly filter regular element and the residue field of $R$ is perfect \cite[Theorem 5.11]{MaQuy}.

By Theorem \ref{THM P->D}, to prove deformation of F-injectivity in full generality it would suffice to prove stability of F-injectivity. In fact, by Remark \ref{Remark strictly filter}, it would suffice to show that F-injectivity is stable for elements which are stricly filter regular elements on $R$.

We are not able to prove stability in general, but we show that it holds in the most relevant cases where deformation of F-injectivity is known to hold (Theorems  \ref{PerturbationFInjective SurjectiveElement} and \ref{Theorem F-inj strictly filter}). 

\subsection{Results on $\m$-adic stability of F-injectivity}
Let $W$ be an $R$-module. We recall that a Frobenius action on $W$ is an additive map $F \colon W \to W$ such that $F(rw) = r^pF(w)$ for all $r \in R$ and $w \in W$. 
We start with an easy but extremely useful lemma.

\begin{Lemma} \label{LemmaInjective}
Let $(R,\m)$ be a local ring, and $c \in \m$ any element. Let $W$ be an Artinian module, with a Frobenius action $F\colon W \to W$. Assume that the map $cF\colon W\to W$ is injective. There exists an integer $N >0$ such that $(c+\delta)F$ is injective for all $\delta \in \m^N$.
\end{Lemma}
\begin{proof}
It suffices to take $N$ such that $\m^{N} \subseteq \m^{[p]}$. In fact, assume that $(c+\delta)F$ was not injective for some $\delta \in \m^{N}$. Then there is a non-zero element $\eta \in \soc(W)$ such that $(c+\delta)F(\eta)=0$. If $\m=(x_1,\ldots,x_n)$, since $\delta \in \m^{N} \subseteq \m^{[p]}$, we can write $\delta = r_1x_1^p + \cdots  + r_nx_n^p$ for some $r_1,\ldots,r_n \in R$. Thus
\[
0=(c+\delta)F(\eta) = cF(\eta) + \sum_{i=1}^n r_i x_i^pF(\eta) = cF(\eta) + \sum_{i=1}^n r_iF(x_i\eta) = cF(\eta),
\]
where $x_i \eta=0$ because $\eta$ is a socle element. This contradicts the assumption that $cF$ is injective and completes the proof.
\end{proof}

\begin{Theorem} \label{PerturbationFInjective SurjectiveElement}
Let $(R,\m)$ be a local ring of dimension $d$, and $x \in \m$ be a surjective element. Assume that $R/(x)$ is F-injective. There exists an integer $N>0$ such that $R/(x+\delta)$ is F-injective for any $\delta \in \m^N$.
\end{Theorem}
\begin{proof}
Since $x$ is a surjective element, for all $i>0$ we have commutative diagrams:
\[
\xymatrix{ 
0 \ar[r] & H^{i-1}_\m(R/(x)) \ar[d]_{\ov{F_{i-1}}} \ar[r] ^-\psi& H^i_\m(R) \ar[d]^{x^{p-1}F_i} \ar[r]^-{\cdot x} & H^i_\m(R) \ar[d]^{F_i} \ar[r]& 0 \\
0 \ar[r] & H^{i-1}_\m(R/(x)) \ar[r]^-\psi & H^i_\m(R) \ar[r]^-{\cdot x} & H^i_\m(R) \ar[r]& 0
}
\] 
We prove that $x^{p-1}F_i$ is injective reproducing the argument in the proof of \cite[Proposition 3.7]{MaQuy}. By assumption we have that $\ov{F_{i-1}}$ is injective for every $i$. Assume that $x^{p-1}F_i$ is not injective, and pick $0 \ne \eta \in \ker(x^{p-1}F_i) \cap \soc(H^i_\m(R))$, so that $x\eta = 0$. By exactness, there exists $\gamma \in H^{i-1}\m(R/(x))$ such that $\psi(\gamma)=\eta$, and then $\psi(\ov{F_{i-1}}(\gamma)) = x^{p-1}F_i(\eta)=0$. Since both $\psi$ and $\ov{F_{i-1}}$ are injective, it follows that $\gamma=0$, and hence $\eta = \psi(\gamma) =0$. A contradiction. So $x^{p-1}F_i$ is injective. By Lemma \ref{LemmaInjective}, Lemma~\ref{Lemma sfr perturbs}, and \cite[Corollary 1.2]{HunekeTrivedi}, we can choose $N>0$ such that $(x+\delta)^{p-1}F_i$ is injective for all $i$, and $x+\delta$ is still a surjective element. Then for all $i$ we still have commutative diagrams
\[
\xymatrix{ 
0 \ar[r] & H^{i-1}_\m(R/(x+\delta)) \ar[d]_{\widetilde{F_{i-1}}} \ar[r] & H^i_\m(R) \ar[d]^{(x+\delta)^{p-1}F_i} \ar[r]^-{\cdot x+\delta} & H^i_\m(R) \ar[d]^{F_i} \ar[r]& 0 \\
0 \ar[r] & H^{i-1}_\m(R/(x+\delta)) \ar[r] & H^i_\m(R) \ar[r]^-{\cdot x+\delta} & H^i_\m(R) \ar[r]& 0
}
\] 
As the middle map is injective, we have that $\widetilde{F_{i-1}}$ is injective for every $i$, and thus $R/(x+\delta)$ is F-injective.
\end{proof}

As an immediate consequence, we obtain $\m$-adic stability of F-injectivity in the Cohen-Macaulay case.

\begin{Corollary} \label{PerturbationF-injectiveCM} Let $(R,\m)$ be a Cohen-Macaulay local ring of dimension $d$, and $x \in \m$ be a non-zero divisor such that $R/(x)$ is F-injective. There exists an integer $N>0$ such that $R/(x+\delta)$ is F-injective for all $\delta \in \m^N$. 
\end{Corollary}
\begin{proof}
Because $R$ is Cohen-Macaulay, the only non-vanishing local cohomology module is $H^d_\m(R)$. Moreover, it follows from \cite[Theorem 7.3.2]{BrodmannSharp} that $\Att(H^d_\m(R)) = \Ass_R(R)$. As $x$ is a non-zero divisor, it is then a surjective element, and we can apply Theorem \ref{PerturbationFInjective SurjectiveElement}.
\end{proof}

We recall that a ring $R$ is called F-pure if the Frobenius map is pure. If $R$ is F-finite, this is equivalent to assuming that the Frobenius map splits (see Section \ref{Section Counterexamples F-reg}).

The following is another immediate corollary of Theorem \ref{PerturbationFInjective SurjectiveElement}.

\begin{Corollary} Let $(R,\m)$ be a local ring and $x \in \m$ be a non-zero divisor such that $R/(x)$ is F-pure. There exists $N > 0$ such that $R/(x+\delta)$ is F-injective for all $\delta \in \m^N$.
\end{Corollary}
\begin{proof}
Since $R/(x)$ is F-pure, $x$ is a surjective element \cite[Proposition 3.5]{MaQuy}, and the corollary now follows from Theorem \ref{PerturbationFInjective SurjectiveElement}.
\end{proof}

We now want to present a result on stability of F-injectivity for strictly filter regular elements. The techniques involved mimic those employed in \cite[Theorem 1.2]{MaQuy} for deformation of F-injectivity, and will eventually require that the residue field is perfect. We start with a useful lemma that holds in greater generality.
\begin{Lemma} \label{Lemma invariance secondary rep for deep perturbations} Let $(R,\m)$ be a local ring, and $x$ be an element of $\m$. Let $W$ be an Artinian $R$-module such that $x \notin \p$ for all $\p \in \Att(W) \smallsetminus \{\m\}$. 
Then there exists an integer $N>0$ such that $xW=(x+\delta)W$ for all $\delta \in \m^{N}$.
\end{Lemma}
\begin{proof}
Let $W = W_1 + \cdots + W_t$ be a minimal secondary representation of $W$, with $\Att(W_t) \subseteq \{\m\}$ (with this we do not exclude the possibility that $\m$ is not attached to $W$, in which case $W_t=0$ and $\Att(W_t) = $ \O). Let $N>0$ be such that $\m^{N} W_t=0$; such an integer exists, because $\ann_R(W_t)$ is either $\m$-primary or the whole ring. By increasing $N$, we may also assume that $x+\delta$ still does not belong to any attached prime of $W$, except possibly $\m$, using the same argument as in the proof of Lemma \ref{Lemma sfr perturbs}. If we let $W'=W_1+\ldots+W_{t-1}$, for any $\delta \in \m^N$ we then have $W'=(x+\delta)W'$. Moreover, we have $xW_t = (x+\delta)W_t$. It follows that 
\[
xW = xW' + xW_t = x(x+\delta)W' + (x+\delta)W_t \subseteq (x+\delta)W.
\]
Equality can then be obtained reversing the roles of $x$ and $x+\delta$.
\end{proof}

The following is the main result of this section. 
\begin{Theorem} \label{Theorem F-inj strictly filter} Let $(R,\m)$ be a local ring with perfect residue field. Assume that $x$ is a strictly filter regular element and that $R/(x)$ is F-injective. There exists an integer $N>0$ such that $R/(x+\delta)$ is F-injective for all $\delta \in \m^N$.
\end{Theorem}
\begin{proof} The short exact sequence $0 \to R \stackrel{\cdot x}{\longrightarrow} R \to R/(x) \to 0$ induces a long exact sequence on local cohomology and the following commutative diagram:
\[
\xymatrix{
\ldots \ar[r] & H^{i-1}_\m(R/(x)) \ar[d]_-{\ov{F}} \ar[r] & H^i_\m(R) \ar[d]^-{x^{p-1}F} \ar[r]^-{\cdot x} & H^i_\m(R) \ar[d]^-{F} \ar[r] & H^i_\m(R/(x)) \ar[d]^-{\ov{F}} \ar[r] & H^{i+1}_\m(R) \ar[d]^-{x^{p-1}F} \ar[r] & \ldots \\ 
\ldots \ar[r] & H^{i-1}_\m(R/(x)) \ar[r] & H^i_\m(R) \ar[r]^-{\cdot x} & H^i_\m(R) \ar[r] & H^i_\m(R/(x)) \ar[r] & H^{i+1}_\m(R) \ar[r] & \ldots
}
\]
Thus, for all $i \in \Z$ we have exact sequences
\[
\xymatrix{
0 \ar[r] & L_{i-1} \ar[r] & H^i_\m(R) \ar[r]^-{\cdot x} & H^i_\m(R) \ar[r] & L_i \ar[r] & 0,
}
\]
where 
\[
L_i=\ker\left(H^i_\m(R/(x)) \to \ann_{H^{i+1}_\m(R)}(x)\right) = \coker(H^i_\m(R)\stackrel{\cdot x}{\longrightarrow} H^i_\m(R)) = H^i_\m(R)/xH^i_\m(R).
\]
Since $x$ is a strictly filter regular element, $L_i$ has finite length for all $i$. Moreover, since the Frobenius action on $H^i_\m(R/(x))$ is injective, so is the restriction to $L_i$. As $R/\m$ is perfect, it follows that the Frobenius action on $L_i$ is actually bijective, and therefore the induced map $\ov{F_i}\colon H^i_\m(R/(x))/L_i \to H^i_\m(R/(x))/L_i$ is again injective for every $i \in \Z$. We then have exact commutative diagrams
\[
\xymatrix{ 
0 \ar[r] & H^{i-1}_\m(R/(x))/L_{i-1} \ar[d]_{\ov{F_{i-1}}} \ar[r] & H^i_\m(R) \ar[d]^{x^{p-1}F_i} \ar[r]^-{\cdot x} & H^i_\m(R) \ar[d]^{F_i} \ar[r] & \ldots\\
0 \ar[r] & H^{i-1}_\m(R/(x))/L_{i-1} \ar[r] & H^i_\m(R) \ar[r]^-{\cdot x} & H^i_\m(R) \ar[r] & \ldots
}
\]
This implies that $x^{p-1}F_i$ is injective for every $i$. Using Lemma \ref{LemmaInjective}, Lemma~\ref{Lemma sfr perturbs} and \cite[Corollary 1.2]{HunekeTrivedi}, we can find an integer $N>0$ such that $x+\delta$ is a strictly filter regular element, and $(x+\delta)^{p-1}F_i$ is injective for all $\delta \in \m^N$ and all $i \in \Z$. By setting $L_i' = \coker(H^i_\m(R) \stackrel{\cdot x+\delta}{\longrightarrow} H^i_\m(R))$ we have that $L_i'$ has finite length, and we have commutative diagrams:
\[
\xymatrix{ 
0 \ar[r] & H^{i-1}_\m(R/(x+\delta))/L_{i-1}' \ar[d]_{\widetilde{F_{i-1}}} \ar[r] & H^i_\m(R) \ar[d]^{(x+\delta)^{p-1}F_i} \ar[r]^-{\cdot x+\delta} & H^i_\m(R) \ar[d]^{F_i} \ar[r]& \ldots \\
0 \ar[r] & H^{i-1}_\m(R/(x+\delta))/L_{i-1}' \ar[r] & H^i_\m(R) \ar[r]^-{\cdot x+\delta} & H^i_\m(R) \ar[r]& \ldots
}
\]
Since the middle map is injective, we conclude that $\widetilde{F_i}\colon H^i_\m(R/(x+\delta))/L_i' \to H^i_\m(R/(x+\delta))/L_i'$ is injective for all $i$. To prove that $R/(x+\delta)$ is F-injective, it suffices to show that the Frobenius action on $L_i'$ is injective for all $i \in \Z$. By Lemma \ref{Lemma invariance secondary rep for deep perturbations} applied to $W=H^i_\m(R)$, there exists an integer $N_i >0$ such that $x H^i_\m(R) = (x+\delta) H^i_\m(R)$ for all $\delta \in \m^{N_i}$. By increasing $N$, if needed, we may assume that $N\geq \max\{N_i \mid i=0,\ldots,d\}$. Then $L_i = H^i_\m(R)/xH^i_\m(R) = H^i_\m(R)/(x+\delta)H^i_\m(R) = L_i'$ for all $\delta \in \m^N$, and since Frobenius is injective on $L_i$, it is injective on $L_i'$ as well.
\end{proof}

\begin{Remark} 
One could ask whether $\m$-adic stability of F-injectivity holds for regular sequences, at least in the assumptions of Theorems \ref{Theorem F-inj strictly filter} and \ref{PerturbationFInjective SurjectiveElement} after appropriately generalizing the notions of surjective element and strictly filter regular element. Our techniques do not immediately show that, and we decided not to pursue this direction in this article.
\end{Remark}

We conclude this section by pointing out that, to the best of our knowledge, there is no example of a local ring $(R,\m)$ and a non-zero divisor $x \in \m$ such that $R/(x)$ is F-injective, and $x$ is not a surjective element. While this may just be due to the difficulty of producing good examples, it is natural to ask whether it might actually be true in general that F-injectivity of $R/(x)$ implies that $x$ is a surjective element. If this was indeed the case, stability and deformation of F-injectivity would both hold in full generality by Theorems~\ref{PerturbationFInjective SurjectiveElement} and \ref{THM P->D}.

\section{Stability of F-regularity and F-purity} \label{Section Counterexamples F-reg}

Let $(R,\m)$ be an F-finite local ring. 
An additive map $\varphi: R \to R$ is called a {\it $p^{-e}$-linear map} if it is additive and $\varphi(r^{p^e} s) = r\varphi(s)$ for all $r,s\in R$. More generally, a {\it Cartier map} is a $p^{-e}$-linear map, for some $e \in \N$.
\begin{Definition} \label{def F-sing}
Let $(R,\m)$ be a local ring of prime characteristic $p>0$, and let $F:R \to R$ denote the Frobenius endomorphism. $R$ is called {\it F-pure} if there exists a $p^{-1}$-linear map $\varphi$ such that $\varphi \circ F = \id_R$. It is called {\it strongly F-regular} if, for all $c \in R$ not in any minimal prime, there exists a $p^{-e}$-linear map $\varphi$ such that $\varphi \circ (cF^e) = \id_R$.
\end{Definition}

As observed in the introduction, $\m$-adic stability of strong F-regularity and F-purity do not hold in general, as a consequence of \cite{Fedder}, \cite{SinghNODeformationFREG}, and Theorem \ref{THM P->D}. Based on the family of examples given in \cite{SinghNODeformationFREG}, we construct an explicit ring where stability of these F-singularities does not hold.

\begin{Example} \cite[Theorem 1.1 with $p=3, m=5$ and $n=2$]{SinghNODeformationFREG} \label{Example Anurag} Let $K$ be a field of characteristic $3$, 
$S=K[a,b,c,d,t]_{\n}$, where $\n=(a,b,c,d,t)$, and $A=S/I$, where $I$ is the ideal generated by the size two minors of
\[
\begin{bmatrix}
a^2+t^5 & b & d \\
c & a^2 & b^2-d
\end{bmatrix}.
\]
Then $A/(t)$ is strongly F-regular, but $A$ is not F-pure.
\end{Example}

\begin{Theorem} \label{Counterexample FReg} Let $K$ be a field of characteristic $3$, $S=K[x,y,z,u,v,w]_\n$, where $\n=(x,y,z,u,v,w)$, and $R=S/I$, where $I$ is the ideal generated by the size two minors of the matrix
\[
\begin{bmatrix}
x^2+v^5 & y & u \\
z & x^2 & y^2-u
\end{bmatrix}.
\]
Then $R/(v)$ is strongly F-regular, but $R/(v-w^N)$ is not F-pure for any $N \geq 1$.

In particular, F-regularity and F-purity are not stable properties in $R$.
\end{Theorem}
\begin{proof}
In the notation of Example \ref{Example Anurag}, we have that $R/(v) \cong \left(A/(t)[W]\right)_{(\m_{A/(t)},W)}$, where $W$ is an indeterminate on $A/(t)$ and $\m_{A/(t)}$ is the maximal ideal of $A/(t)$. Since adding a variable and then localizing does not affect strong F-regularity, $R/(v)$ is strongly F-regular.
For a given $N \geq 1$, we let $S_N=K[x,y,z,u,v,w^N]_{\n_N}$, where $\n_N = (x,y,z,u,v,w^N)$. Observe that $S_N \to S$ is a faithfully flat map. If we let $I_N = I + (v-w^N)$, $T_N=S_N/I_N$ and $R_N=S/I_N = R/(v-w^N)$, then $T_N \to  R_N$ is also faithfully flat. It is easy to see that $T_N \cong A$, and thus $T_N$ is not F-pure for any $N \geq 1$. By \cite[Proposition 5.13]{HochsterRoberts} we then conclude that $R_N$ is not F-pure either. In conclusion, $R/(v)$ is strongly F-regular, hence F-pure. On the other hand, $R_N = R/(v-w^N)$ is not F-pure for any integer $N \geq 1$. A fortiori, it is not strongly F-regular for any $N \geq 1$.
\end{proof}

Patakfalvi, Schwede, and Zhang have previously considered the behavior of strong F-regularity in a family and in
\cite[Corollary~4.19]{PSZ} proved a form of openness for proper maps. 
However, essentially the same example shows that strong F-regularity is not generally open in families,
and it follows that the F-signature is not lower semicontinuous in families either. This is another difference with the Hilbert--Kunz multiplicity (\cite{SmirnovFamily}). 

\begin{Example}\label{Example Family}
Consider the family $\overline{\F_3}[t] \to R[t]/(v - tw)$ where $R$ is the ring of Theorem \ref{Counterexample FReg} with $K=\ov{\F_3}$.
Observe that for $t \neq 0$ then the closed fiber is isomorphic to $A$ from Example~\ref{Example Anurag}, thus $t = 0$ is the only 
strongly F-regular fiber.  Note that all fibers are F-rational because F-rationality deforms and, thus, $A$ is F-rational. 
\end{Example}

This example also shows that the analogue of \cite[Theorem~7.5]{MaSchwede} does not hold for strong F-regularity.
In fact, if we let $D = \overline{\F_3}[t]$ and $Q = (x,y,z,u,v,w)$, then for every maximal ideal $\n_R = (t - r)$ of $D$, the ideal $Q_r = \n_r + Q \in \Spec(R)$ is prime but the set 
\[
\{r \in \overline{\F_3} \mid (R/\n_r)_{Q_r} \text{ is strongly F-regular}\} = \{0\}
\]
is not open.

Working on the same ring, we can also find an example of a proper (in fact, projective) morphism whose special fiber is strongly F-regular, but all other closed fibers are not. It is now known that F-rationality is open in families (\cite[Theorem~5.8]{Hashimoto01}, \cite[Theorem~5.13]{PSZ}), but only special cases were known for strong F-regularity 
(\cite[Corollaries~4.19,4.20]{PSZ}). 

\begin{Example}\label{Proper Family}
Consider the ring $R=\overline{\F_3}[t,s,v,w,x,y,z,u]/I$, with
\[
I=\left(I_2\begin{bmatrix} x^2+v^5 & y & u \\ z & x^2 & y^2-u \end{bmatrix}, v-tw\right).
\]
We observe that $R$ is graded with respect to the following degrees of the variables:
\[
\begin{cases}
\deg(t)=0 \\
\deg(s)=1 \\
\deg(v)=\deg(w)=2 \\
 \deg(x)=5 \\ \deg(y)=\deg(z)=10 \\
 \deg(u)=20
\end{cases}
\]
We let $X=\Proj(R)$, and we consider the natural morphism $f\colon X \to \Spec(\overline{\F_3}[t])$, which is projective. For $r \in \overline{\F_3}$, we let $X_r$ be the fiber over $Q_r=(t-r) \in \Spec(\overline{\F_3})$. We claim that $X_r$ is strongly F-regular if and only if $r=0$.
We start by showing that $X_0$ is strongly F-regular. Observe that
\[
X_0 = \Proj\left(\frac{\overline{\F_3}[s,x,y,z,u,v,w]}{\left(I_2\begin{bmatrix} x^2+v^5 & y & u \\ z & x^2 & y^2-u \end{bmatrix}, v \right)} \right) = \Proj \left(\frac{\overline{\F_3}[s,x,y,z,u,w]}{\left(I_2\begin{bmatrix} x^2 & y & u \\ z & x^2 & y^2-u \end{bmatrix}\right)} \right).
\]
We check that the ring we obtain on each standard affine chart is strongly F-regular. Since $R$ is not standard graded, the computation of the affine charts is not completely straightforward, therefore we show part of the argument.
\begin{itemize}
\item On the chart $U_s=\{s \ne 0\}$, we let
\[
X=\frac{x}{s^5}, \quad Y= \frac{y}{s^{10}}, \quad Z = \frac{z}{s^{10}}, \quad U=\frac{u}{s^{20}}, \quad W= \frac{w}{s^2}.
\]
Observe that $X,Y,Z,U,W$ all have degree zero. The affine chart $U_s$ is then equal to $\Spec(R_s)$, where
\[
R_s = \frac{\overline{\F_3}[X,Y,Z,U,W]}{\left(I_2\begin{bmatrix} X^2 & Y & U \\ Z & X^2 & Y^2-U \end{bmatrix}\right)}.
\]
Since $R_s$ is isomorphic to one of the rings $A/(t)$ of Example \ref{Example Anurag} with the addition of the variable $W$, it is strongly F-regular.
\item On the chart $U_x = \{x \ne 0\}$, we let $\alpha$ be a fifth root of $x$, i.e., $\alpha^5=x$. In this way, $\deg(\alpha)=1$. We let
\[
S=\frac{s}{\alpha}, \quad Y= \frac{y}{\alpha^{10}}, \quad Z = \frac{z}{\alpha^{10}}, \quad U=\frac{u}{\alpha^{20}}, \quad W= \frac{w}{\alpha^2}.
\]
Observe that $S,Y,Z,U, W$ all have degree zero. We then have that $U_x=\Spec(R_x)$, where $R_x$ is the ring of invariants of
\[
S_x=\frac{\overline{\F_3}[S,Y,Z,U,W]}{\left(I_2\begin{bmatrix} 1 & Y & U \\ Z & 1 & Y^2-U \end{bmatrix}\right)}
\]
under the action
\[
\begin{cases} S \mapsto \epsilon_5^{-1}S = \epsilon_5^4S \\
Y \mapsto \epsilon_5^{-10}Y = Y \\ 
Z \mapsto \epsilon_5^{-10}Z=Z \\
 U \mapsto \epsilon_5^{-20}U =U \\ W \mapsto \epsilon_5^{-2}W = \epsilon_5^3W
\end{cases}
\]
where $\epsilon_5$ is a primitive fifth root of $1$. In particular, since the characteristic $p=3$ does not divide $5$, we have that $R_x$ is a direct summand of $S_x$. Thus, it suffices to show that $S_x$ is strongly F-regular. To this end, observe that 
\[
\left(I_2\begin{bmatrix} 1 & Y & U \\ Z & 1 & Y^2-U\end{bmatrix}\right) = (YZ-1,U-Y^2+ZU,U+YU-Y^3).
\]
Multiplying the second equation by $Y$, and using that $YZ=1$ in the quotient, shows that the third equation is redundant. With similar considerations, one can show that the defining ideal of $S_x$ becomes equal to $(UZ^2+UZ^3-1,Y-UZ-UZ^2)$, and thus
\[
S_x = \frac{\overline{\F_3}[S,Y,Z,U,W]}{(UZ^2+UZ^3-1,Y-UZ-UZ^2)} \cong \frac{\overline{\F_3}[S,Z,U,W]}{(UZ^2+UZ^3-1)},
\]
which can be checked to be regular, using for instance the Jacobian criterion. Since regular rings are strongly F-regular, this shows that $U_x$ is strongly F-regular.
\item With computations analogous to those done above for $U_x$ one can check that the affine chart $U_y = \{y \ne 0\}$ is $\Spec(R_y)$, where $R_y$ is a direct summand of the regular ring
\[
S_y = \frac{\overline{\F_3}[S,X,Z,U,W]}{(Z-X^4,1-U-X^2U)} \cong \frac{\overline{\F_3}[S,X,U,W]}{(1-U-X^2U)}.
\]
\item The affine chart $U_z = \{z \ne 0\}$ is $\Spec(R_z)$, where $R_z$ is a direct summand of the regular ring
\[
S_y = \frac{\overline{\F_3}[S,X,Y,U,W]}{(Y-X^4, U+X^2U-X^{10})} \cong \frac{\overline{\F_3}[S,X,U,W]}{(U+X^2U-X^{10})}.
\]
\item The affine chart $U_u = \{u \ne 0\}$ is $\Spec(R_u)$, where $R_u$ is a direct summand of the regular ring
\[
S_u = \frac{\overline{\F_3}[S,X,Y,Z,W]}{(Z+X^2-X^2Y^2,Y+X^2-Y^3)} \cong \frac{\overline{\F_3}[S,X,Y,W]}{(Y+X^2-Y^3)}
\]
\item Finally, the affine chart $U_w = \{w \ne 0\}$ is $\Spec(R_w)$, where $R_w$ is a direct summand of
\[
S_w = \frac{\overline{\F_3}[S,X,Y,Z,U]}{\left(I_2\begin{bmatrix} X^2 & Y & U \\ Z & X^2 & Y^2-U \end{bmatrix}\right)}.
\]
As in the case of $U_s$, we have that $S_w$ is strongly F-regular.
\end{itemize}
To conclude the proof, we need to show that 
\[
X_r = \Proj\left( \frac{\overline{\F_3}[s,x,y,z,u,v,w]}{\left(I_2\begin{bmatrix} x^2+v^5 & y & u \\ z & x^2 & y^2-u\end{bmatrix},v-rw\right)} \right) = \Proj\left( \frac{\overline{\F_3}[x,y,z,u,w]}{\left(I_2\begin{bmatrix} x^2+r^5w^5 & y & u \\ z & x^2 & y^2-u\end{bmatrix}\right)}\right)
\]
is not strongly F-regular for any $r \ne 0$. We consider the affine chart $U_s = \{s \ne 0\}$. Since $\deg(s)=1$, this is equal to $\Spec(R_s)$, where 
\[
R_s = \frac{\overline{\F_3}[X,Y,Z,U,W]}{\left(\begin{bmatrix} X^2+(rW)^5 & Y & U \\ Z & X^2 & Y^2-U \end{bmatrix}\right)}.
\]
Observe that $R_s$ is isomorphic to the ring $A$ of Example \ref{Example Anurag}, hence it is not strongly F-regular.
\end{Example}

\begin{Remark}
Essentially the same example works for mixed characteristic: 
we may set $R=\mathbb Z[s,v,w,x,y,z,u]/I$, with
\[
I=\left(I_2\begin{bmatrix} x^2+v^5 & y & u \\ z & x^2 & y^2-u \end{bmatrix}, v-3w\right).
\]
Then the natural morphism $f\colon X \to \Spec(\mathbb Z)$ gives a family which is strongly F-regular over the prime $(3)$, but not strongly F-regular over any prime larger than $3$.
\end{Remark}

Such examples shows the necessity of a $\Q$-Gorenstein assumption: good behavior is known for F-rational singularities (\cite[Theorem~7.2]{MaSchwede}) 
and strongly F-regular $\Q$-Gorenstein singularities (\cite[Theorem~7.9]{MaSchwede} and \cite[Proposition~9.3]{MSTWW} with milder assumptions). We will return to this later in this section.

The ring of Theorem \ref{Counterexample FReg} gives also an example of a Cohen-Macaulay local domain for which the F-signature is not continuous in the $\m$-adic topology. In fact, neither the F-splitting ratio nor the splitting dimension in the example are continuous. We recall the definitions first.

\begin{Definition}
Let $(R,\m,k)$ be an F-finite $d$-dimensional local ring of characteristic $p>0$. Let $\alpha(R) = \log_p[k:k^p]$. For every integer $e$, write $F^e_*R = R^{\oplus a_e} \bigoplus M_e$, where $M_e$ has no free summands. 
\begin{itemize}
\item The {\it F-signature} of $R$ is $\s(R) = \ds \lim_{e \to \infty} \frac{a_e}{p^{e(\dim(R)+\alpha(R))}}$

\item The {\it splitting dimension} of $R$ is defined as 
\[
\sdim(R) = \left\{\begin{array}{ll} -1 & \text{ if } $R$ \text{ is not F-pure} \\ \\
\max\left\{s \in \Z \ \bigg| \ \lim_{e \to \infty} \frac{a_e}{p^{e(s+\alpha(R))}}>0\right\} & \text{ otherwise}
\end{array} \right.
\]

\item The {\it F-splitting ratio} of $R$ is defined as $r_F(R) = \ds \lim_{e \to \infty} \frac{a_e}{p^{e(\sdim(R)+\alpha(R))}}$.
\end{itemize}
\end{Definition}
\begin{Corollary} \label{CorollaryNOContinuityFSig} Let $(R,\m)$ be an F-finite local ring. The functions $f \in R \mapsto \sdim(R/(f))$, $f \in R \mapsto r_F(R/(f))$ and $f \in R \mapsto \s(R/(f))$ are not continuous in the $\m$-adic topology.
\end{Corollary}
\begin{proof}
We use the notation of Theorem \ref{Counterexample FReg}. By localizing $R$ at the irrelevant maximal ideal, we have produced an example of a ring such that $\sdim(R/(v)) = \dim(R/(v))=3$, but $\sdim(R/(v-w^N))=-1$ for all $N \geq 1$. Moreover, we have that $\s(R/(v)) = r_F(R/(v)) >0$, but $\s(R/(v-w^N)) = r_F(R/(v-w^N)) =0$ for all $N \geq 1$.
\end{proof}

In \cite{PolstraSmirnov}, Polstra and the second author prove that the F-signature function is continuous for Gorenstein rings. This follows also from the results of this paper, since for a Gorenstein local ring F-signature and F-rational signature are equal \cite[Lemma 4.3]{HochsterYao}. Note that strong F-regularity and F-rationality are also equivalent for a Gorenstein local ring. As Corollary \ref{CorollaryNOContinuityFSig} makes evident, the F-signature behaves very differently from the Hilbert-Kunz multiplicity in this context. In fact, under mild assumptions the Hilbert--Kunz multiplicity is $\m$-adically continuous for a Cohen-Macaulay local ring \cite[Theorem 1.3]{PolstraSmirnov}.

\subsection{$\Q$-Gorenstein rings}

We recall that a Cohen-Macaulay normal local domain $(R,\m)$ with a canonical ideal $J$ is called {\it $\Q$-Gorenstein} if there exists an integer $n >0$ such that the symbolic power $J^{(n)}$ is principal. Equivalently, $J$ is a torsion element when seen as an element in the class group of $R$.

By \cite{AKM}, if $R$ is $\Q$-Gorenstein and $R/xR$ is weakly F-regular, then $R$ is weakly F-regular.
It is also known that $\Q$-Gorenstein weakly F-regular rings are strongly F-regular (\cite[Theorem~3.3.2]{MacCrimmonThesis}, see also \cite[(2.2.4)]{AberbachWeakAndStrong}). These facts imply that strong F-regularity deforms if $R$ is $\Q$-Gorenstein. In this subsection, we prove an analogous statement for $\m$-adic stability.

\begin{Lemma}\label{Lemma symbolic inclusion}
Let $(R,\m)$ be a Cohen-Macaulay local ring of dimension $c+d$, and $\ul{x} = x_1,\ldots,x_c$ be a regular sequence such that $\ov{R} = R/(\ul{x})$ is a normal domain.
Let $J \subseteq R$ be a canonical ideal of $R$ such that $J\ov{R}$ is a canonical ideal of $\ov{R}$.
Then $J^{(n)}\ov{R} \subseteq (J\ov{R})^{(n)}$ for all $n \in \N$.
\end{Lemma}
\begin{proof}
The claim is trivial if $J^{(n)} = J^n$. Assume henceforth that this is not the case, that is, $J^n$ has embedded components. First, we prove that no associated prime of $J^n$ of height $c+1$ can be minimal over $(J,\ul{x})$. In fact, let $P \in \Spec(R)$ be a prime, with $\height(P)=c+1$, and assume that $P$ is minimal over $(J,x)$. Since $R/(\ul{x})$ is normal, $\left(R/(\ul{x})\right)_P$ is regular, hence Gorenstein. It follows that $R_P$ is Gorenstein, and thus $J_P$ is a principal ideal of $R_P$. Therefore $(J^n)_P = (J_P)^n$ is an ideal of height one with no embedded components, and thus $P$ is not an associated prime of $J^n$. Now let $f \in R$ be an element that avoids all minimal primes of $J$, and belongs to the embedded components of $J^n$. Because of what we have shown above, by prime avoidance we may also choose $f$ so that it does not belong to any minimal prime of $(J,x)$. We then have $J^{(n)} = (J^n:_R f)$. Taking residue classes modulo $(\ul{x})$, we obtain that $J^{(n)}\ov{R} = (J^n\ov{R}:_{\ov{R}} f) \subseteq (J\ov{R})^{(n)}$, where the last containment follows from the fact that the class of $f$ modulo $(\ul{x})$ does not belong to the minimal primes of $J\ov{R}$, by our choice of $f$.
\end{proof}

The following is a modification of \cite[Lemma 6.7 (i.)]{PolstraTucker}.

\begin{Corollary} \label{Cor PT} Let $(R,\m)$ be an F-finite Cohen-Macaulay local ring of dimension $c+d$ and characteristic $p > 0$. Let $\ul{x} = x_1,\ldots,x_c$ be a regular sequence such that $\ov{R} = R/(\ul{x})$ is a normal domain and let $J \subseteq R$ be an ideal of height $c+1$ such that $J\ov{R}$ is the canonical module of $\ov{R}$. Let $y_1 \in J$ and $\ul{y} = y_1,y_2,\ldots,y_d \in R$ be such that $\ul{x},\ul{y}$ forms a regular sequence. If $y_2 J \ov{R} \subseteq a_2\ov{R}$ for some $a_2 \in J$, then
for all $t \geq 2$ and all $e \geq 0$
\begin{align*}
& (J^{(p^e)},\ul{x},y_2^{tp^e},y_3^{tp^e},\ldots,y_d^{tp^e}): y_2^{(t-1)p^e} \\
= &(J^{[p^e]},\ul{x},y_2^{tp^e},y_3^{tp^e},\ldots,y_d^{tp^e}): y_2^{(t-1)p^e}\\
= & (J^{[p^e]},\ul{x},y_2^{2p^e},y_3^{tp^e},\ldots,y_d^{tp^e}): y_2^{p^e}
\end{align*}
\end{Corollary}
\begin{proof} 
Using \cite[Lemma 6.7]{PolstraTucker} in $\ov{R}$, we get that 
\[
((J\ov{R})^{(p^e)},y_2^{tp^e},y_3^{tp^e},\ldots,y_d^{tp^e}):_{\ov{R}} y_2^{(t-1)p^e}
= (J^{[p^e]},y_2^{2p^e},y_3^{tp^e},\ldots,y_d^{tp^e})\ov{R} :_{\ov{R}} y_2^{p^e}.
\]
By Lemma~\ref{Lemma symbolic inclusion}, we then get that the first term of the assertion 
is contained in the last, but the opposite inclusion is clear. 
\end{proof}

\begin{Theorem} \label{Theorem Q-Gorenstein}
Let $(R, \m)$ be an F-finite local ring of characteristic $p > 0$ that is $\Q$-Gorenstein on the punctured spectrum of $R$. 
Let $\ul{x}$ be a part of a system of parameters and suppose that $R/\ul{x}R$ is strongly F-regular. 
Then for any $\varepsilon > 0$ there exists an integer $N > 0$
such that for all $\ul{\delta} \in (\m^N)^{\oplus c}$ 
\[
\left| \s(R/(\ul{x}+\ul{\delta})) - \s(R/(\ul{x})) \right| < \varepsilon.
\]

In particular,  there exists an integer $N > 0$
such that for all $\ul{\delta} \in (\m^N)^{\oplus c}$ the ring $R/(\ul{x}+\ul{\delta})$ is strongly F-regular.
\end{Theorem}
\begin{proof}
Since $R$ is F-rational, it is a Cohen-Macaulay normal domain. Let $J$ be a canonical ideal of $R$. We can find $N \gg 0$ such that $\ul{x}+\ul{\delta}$ is a regular sequence for all $\delta \in (\m^N)^{\oplus c}$. After increasing $N$ if needed, by Corollary~\ref{F-rational is ok} we may assume that $R/(\ul{x}+\ul{\delta})$ is F-rational and, therefore, normal. Moreover, we can ensure that $\frac{J+(\ul{x}+\ul{\delta})}{(\ul{x}+\ul{\delta})}$ is a canonical ideal of $R/(\ul{x}+\ul{\delta})$ by \cite{PolstraSmirnov}. By applying the proof of \cite[Corollary~5]{AKM} in localizations, we see that the quotient $R/(\ul{x}+\ul{\delta})$ is $\Q$-Gorenstein on the punctured spectrum for any $\ul{\delta} \in (\m^N)^{\oplus c}$. In particular, $R/(\ul{x})$ is $\Q$-Gorenstein on the punctured spectrum.

We now mimic the proof of \cite[Corollary 6.8 (ii.)]{PolstraTucker}, using \cite[Lemma 6.7 (ii.)]{PolstraTucker} and Corollary~ \ref{Cor PT} in place of \cite[Lemma 6.7 (i.)]{PolstraTucker}. Let $0 \ne y_1 \in J$ be such that $\ul{x},y_1$ is a regular sequence. Let $U_2$ be the complement of the minimal primes of $(y_1)\ov R$, and observe that $U_2^{-1}\ov{J}$ is principal, since $\ov R$ is normal. We can then find $y_2 \in U_2$ such that $y_2J\ov R \subseteq a_2\ov R$ for some $a_2 \in J$. Now, let $U_3$ be the complement of the minimal primes of $\ul{x},y_1,y_2$. Since $R$ is $\Q$-Gorenstein on $\Spec(R) \smallsetminus \{\m\}$, there exists $n$ and $y_3 \in U_3$ such that $y_3^nJ^{(n)} \subseteq a_3$ for some $a_3 \in J^{(n)}$. Repeating this process, we obtain a system of parameters $\ul{x},\ul{y}$ such that $y_i^nJ^{(n)} \subseteq a_iR$ for all $i=3,\ldots,d$, with $a_i \in J^{(n)}$, and $y_2J \subseteq (a_2,\ul{x})$ for some $a_2 \in J$. After increasing $N$, we may assume that $(J,\ul{x},\ul{y}) = (J,\ul{x}+\ul{\delta},\ul{y})$ for all $\ul{\delta} \in (\m^N)^{\oplus c}$. In particular, $\ul{x}+\ul{\delta},\ul{y}$ is still a system of parameters in $R$. Let $u$ be an element of $(J,\ul{x},\ul{y}):\m$ which does not belong to $(J,\ul{x},\ul{y})$. Observe that, by our choice of $N$, we have that $u R/(J,\ul{x}+\ul{\delta},\ul{y})$ generates the socle of $R/(J,\ul{x}+\ul{\delta},\ul{y})$ for all $\ul{\delta} \in (\m^N)^{\oplus c}$. By a repeated application of \cite[Lemma 6.7 (ii.)]{PolstraTucker} and by Corollary~\ref{Cor PT}, for all $t\geq 2$ and $e \geq 0$, and for all $\ul{\delta} \in (\m^N)^{\oplus c}$ we have
\begin{align*}
& (J^{[p^e]},\ul{x}+ \ul{\delta},y_2^{tp^e},y_3^{tp^e},\ldots,y_d^{tp^e}): u^{p^e}y_2^{(t-1)p^e} \cdots y_d^{(t-1)p^e} \\
\subseteq & (J^{(p^e)},\ul{x}+ \ul{\delta},y_2^{tp^e},y_3^{tp^e},\ldots,y_d^{2p^e}): u^{p^e}y_2^{(t-1)p^e} y_3^{(t-1)p^e}\cdots y_d^{p^e}y_1^n \\
& \hspace{5cm}\vdots \\
\ldots  \subseteq  &  (J^{(p^e)},\ul{x}+ \ul{\delta},y_2^{tp^e},y_3^{2p^e},\ldots,y_d^{2p^e}): u^{p^e}y_2^{(t-1)p^e} y_3^{p^e}\cdots y_d^{p^e}y_1^{(d-1)n} \\
 = &  (J^{[p^e]},\ul{x}+ \ul{\delta},y_2^{2p^e},\ldots,y_d^{2p^e}):(u y_2 \cdots y_d)^{p^e} y_1^{(d-1)n}.
\end{align*}

Let $I = (J^{[p^e]},y_2^{2p^e},\ldots,y_d^{2p^e})$ and $u=\delta y_2\cdots y_d$.
By \cite[Lemma 4.13]{PolstraTucker}, for all $\ul{\delta} \in (\m^N)^{\oplus c}$ we may compute
\[
\s (R/(\ul{x}+\ul{\delta})) = \ehk(I;R/(\ul{x}+\ul{\delta})) - \ehk((I,z);R/(\ul{x}+\ul{\delta})).
\]
As observed in Theorem \ref{very uniform convergence}, there is $N \gg 0$ such that 
for all $\ul{\delta} \in (\m^N)^{\oplus c}$ we have
\[
\left|\ehk(I;R/(\ul{x}+\ul{\delta})) - \ehk(I;R/(\ul{x})) \right| < \frac{\varepsilon}{2}
\]
and 
\[
\left|\ehk((I,z);R/(\ul{x}+\ul{\delta})) - \ehk((I,z);R/(\ul{x})) \right| < \frac{\varepsilon}{2}.
\]
Therefore
$\left| \s(R/(\ul{x}+\ul{\delta})) - \s(R/(\ul{x})) \right| < \varepsilon$.

The second assertion follows by taking $\varepsilon = \s(R/(\ul{x}))/2$.
\end{proof}

\begin{Remark}
We cannot use Theorem~\ref{THM P->D} to deduce from Theorem~\ref{Theorem Q-Gorenstein} that 
F-regularity deforms for rings $\Q$-Gorenstein on the punctured spectrum. 
In fact, Example~\ref{Example Anurag} is $\Q$-Gorenstein on the punctured spectrum, since 
it is an $F$-rational three-dimensional ring (e.g., by \cite[Proposition 17.1]{Lipman} or \cite[7.3.2]{Ishii}). 
The same example also shows that it is not enough to require that $R/(\ul{x})$ is $\Q$-Gorenstein in Theorem~\ref{Theorem Q-Gorenstein},
since even $\Q$-Gorenstein strongly F-regular singularities do not deform. 
\end{Remark}

On the other hand, a positive result can be obtained if we assume that $R$ is Gorenstein in sufficiently high codimension.

\begin{Corollary}
Let $(R, \m)$ be an F-finite local ring of characteristic $p > 0$. 
Let $\ul{x}$ be a part of a system of parameters and suppose that $R/(\ul{x})$ is $\Q$-Gorenstein and strongly F-regular. 
Suppose that $R$ and $R/(\ul{x})$ are Gorenstein in codimension two. 
Then, there exists an integer $N > 0$
such that for all $\ul{\delta} \in (\m^N)^{\oplus c}$ the ring $R/(\ul{x}+\ul{\delta})$ is strongly F-regular.
\end{Corollary}
\begin{proof}
We use \cite[Proposition~9.1.9]{Ishii} to show that $R$ is $\Q$-Gorenstein, and then invoke Theorem \ref{Theorem Q-Gorenstein}.
\end{proof}

\subsection{Positive results on stability of F-purity} \label{Section perturbation F-pure}

In light of Section \ref{Section Counterexamples F-reg}, there is no hope to prove that $\m$-adic stability of F-purity holds, even under strong assumptions such as the ring being Cohen-Macaulay. In fact, by Theorem~\ref{THM P->D}, one needs to ensure that deformation of F-purity holds, at the very least. To this end, we make the following definition.

\begin{Definition} Let $(R,\m)$ be an F-finite local ring, and $\a \subseteq \m$ be an ideal. We say that $R$ is compatibly F-pure along $\a$ if there exists a Cartier map $\varphi \colon R \to R$ such that $\varphi(1)=1$ and $\varphi(\a) \subseteq \a$.
\end{Definition}

Equivalently, $R$ is compatibly F-pure along $\a$ if $R$ is F-pure, and there is a splitting of Frobenius that descends to $R/\a$. In this sense, when $\a=(x)$ is generated by a regular element, we can view compatible F-purity along $\a$ as a strong version of deformation of F-purity. The goal of this section is to show that stability of F-purity holds if $R$ is compatibly F-pure along an ideal generated by a regular element. 

We first need some general results on Cartier maps. When $R$ is F-finite, we can write it as a quotient $R=S/I$ of an F-finite regular local ring $(S,\n)$ by \cite{Gabber}. Fedder's criterion establishes that $R$ is F-pure if and only if $(I^{[p]}:_S I) \not\subseteq \n^{[p]}$. Moreover, it is well-known that $p^{-e}$-linear maps on $R$ correspond to elements of the ideal $(I^{[p^e]}:_S I)$ in the following way: if $\Tr\colon S \to S$ denotes the trace map on $S$, then every $p^{-e}$-linear map on $R$ is of the form $\Tr^e(s \cdot -)$ for some $s \in I^{[p^e]}:_S I$, and viceversa. The correspondence is one-to-one modulo $I^{[p]}$. A surjective $p^{-e}$-linear map corresponds to an element $\alpha \in (I^{[p^e]}:_S I)$ that does not belong to $\n^{[p^e]}$.

From this point of view, the ring $R=S/I$ is compatibly F-pure along $R/\a = S/(I,\a)$ if and only if there exists $\alpha \in (I^{[p]}:_S I) \cap ((I,\a)^{[p]}:_S (I,\a))$, with $\alpha \notin \n^{[p]}$.

\begin{Theorem} \label{Perturbation FPurity Compatible} Let $(R,\m)$ be an F-finite local ring and $x \in \m$ be a regular element on $R$. Assume that $R$ is compatibly F-pure along $(x)$. There exists an integer $N \gg 0$ such that $R$ is compatibly F-pure along $(x+\delta)$ for all $\delta \in \m^N$. In particular, $R/(x + \delta)$ is F-pure for each $\delta \in \m^N$.
\end{Theorem}
\begin{proof}

Write $R=S/I$, where $(S,\n)$ is a regular local ring. Since $R$ is compatibly F-pure along $(x)$, there exists $\alpha \in (I^{[p]}:_S I) \cap ((I,x)^{[p]}:_S (I,x))$ with $\alpha \notin \n^{[p]}$. Since $x\alpha  \in (I,x)^{[p]}$, there exists $\beta \in R$ such that $x(\alpha - x^{p-1}\beta) \in I^{[p]}$. Since $x$ is a non-zero divisor on $R=S/I$, and $S$ is regular, $x$ is a non-zero divisor also on $S/I^{[p]}$. In particular, $\alpha - x^{p-1}\beta \in I^{[p]}$. We now claim that $\beta \in I^{[p]}:_S I$. In fact, since $\alpha \in I^{[p]}:_S I$, and because $\alpha -x^{p-1}\beta  \in I^{[p]}$, we conclude that $x^{p-1}\beta I \subseteq I^{[p]}$. As above, $x^{p-1}$ is a non-zero divisor on $S/I^{[p]}$, and thus $\beta I \subseteq I^{[p]}$, as claimed. Finally, observe that $x^{p-1}\beta$ is congruent to $\alpha$ modulo $\n^{[p]}$, and thus $x^{p-1}\beta \notin \n^{[p]}$. In particular, we have that $(x+\delta)^{p-1} \beta \notin \n^{[p]}$ for all $\delta \in \n^{[p]}$. Choose $N \gg 0$ such that $\n^N \subseteq \n^{[p]}$, and such that $x+\varepsilon$ is regular on $R$ for all $\varepsilon \in \n^N$. For any given $\delta \in \m^N$, pick a representative modulo $I$ that belongs to $\n^N$. Then $(x+\delta)^{p-1}\beta \in (I^{[p]}:_SI) \cap ((I,x+\delta)^{[p]}:_S (I,x+\delta))$, and $(x+\delta)^{p-1}\beta \notin \n^{[p]}$. It follows that $R$ is compatibly F-pure along $(x+\delta)$.
\end{proof}

The proof of Theorem \ref{Perturbation FPurity Compatible} actually shows that if $R$ is compatibly F-pure along $(x)$, then the pair $(R,x)$ is sharply F-pure.

\begin{Definition} \cite[Definition 3.1]{SchwedeSharpFPurity} Let $(R,\m)$ be an F-finite local ring, $x \in \m$ a regular element and $t \in \R_{\geq 0}$. The pair $(R,x^t)$ is {\it sharply F-pure} if there exist Cartier maps $\varphi_e$ such that $\varphi_e \circ (x^{\lceil t(p^e-1) \rceil}F^e )= \id_R$ for infinitely many $e \geq 0$.
\end{Definition}

The notion of sharp F-purity introduced by Schwede in \cite{SchwedeSharpFPurity} is much more general than the one above. Since we are only interested in pairs of this form, we will not recall the more general version here.

One can show that being compatibly F-pure along $(x)$ is equivalent to the pair $(R,x)$ being sharply F-pure. Following the proof of Theorem \ref{Perturbation FPurity Compatible}, we obtain the following stability result for sharp F-purity of pairs. 

\begin{Proposition} \label{PropositionSharpFPurity} Let $(R,\m)$ be an F-finite local ring, $x \in \m$ a regular element and $t \in \R_{\geq 0}$. Assume that $(R,x^t)$ is sharply F-pure. There exists $N \gg 0$ such that $(R,(x+\delta)^t)$ is sharply F-pure for all $\delta \in \m^N$.
\end{Proposition}
\begin{proof}
By the correspondence noted above, if we write $R=S/I$ for some regular local ring $(S,\n)$, then $(R,x^t)$ is sharply F-pure if and only if there exist elements $\alpha_e \in (I^{[p^e]}:I)$ such that $x^{\lceil t(p^e-1) \rceil}\alpha_e \notin \n^{[p^e]}$ for infinitely many values of $e \geq 0$. However,  \cite[Proposition 3.3]{SchwedeSharpFPurity}  shows that it is 
enough to find a single $e_0$ for which this is the case.

Let $e_0 >0$ and an element $\alpha_0 \in (I^{[p^{e_0}]}:I)$ be such that $x^{\lceil t(p^{e_0}-1)\rceil}\alpha_{e_0} \notin \n^{[p^{e_0}]}$. Choose $N$ such that $x+\delta$ is regular for all $\delta \in \m^N$, and such that $\m^N \subseteq \m^{[p^{e_0}]}$. We have $(x+\delta)^{\lceil t(p^{e_0}-1)\rceil}\alpha_{e_0} = x^{\lceil t(p^{e_0}-1)\rceil}\alpha_{e_0} + \delta \beta$ for some $\beta \in R$. As $\delta \in \m^{[p^{e_0}]}$, it follows that $(x+\delta)^{\lceil t(p^{e_0}-1)\rceil}\alpha_{e_0} \notin \n^{[p^{e_0}]}$, and then $(R,(x+\delta)^t)$ is sharply F-pure by \cite[Proposition 3.3]{SchwedeSharpFPurity}.
\end{proof}

It is also natural to ask whether the strategy used in the proof of Theorem \ref{Perturbation FPurity Compatible} extends to regular sequences in place of just regular elements. The next example shows that this is not the case.

\begin{Example} Let $S=\F_3[x,y,z,u,v,w,t]_\n$, where $\n$ is the maximal ideal generated by the variables of $S$. Let $I$ be the ideal generated by the $2 \times 2$ minors of
\[
\begin{bmatrix}
x^2+v^5 & y & u \\
z & x^2 & t-u
\end{bmatrix},
\]
and set $T=S/I$. Let $x_1=t-u^2$ and $x_2=v$, and let $R$ be the ring constructed in Theorem \ref{Counterexample FReg}. It can be checked using Macaulay2 (\cite{M2}) that:
\begin{itemize}
\item $x_1,x_2$ forms a regular sequence on $T$.
\item $T$ is F-pure and $T/(x_1,x_2) \cong R/(v)$ is F-pure. In fact, a stronger statement is true: $T$ is compatibly F-pure along $(x_1,x_2)$.
\item For all integers $N \geq 1$, $T$ is not compatibly F-pure along $(x_1,x_2-w^N)$. In fact, $T/(x_1,x_2-w^N)\cong R/(v-w^N)$ is not F-pure, as shown in Theorem \ref{Counterexample FReg}.
\end{itemize}
\end{Example}
In order to obtain a stability result for regular sequences $x_1,\ldots,x_c$ with the strategy of the proof of Theorem \ref{Perturbation FPurity Compatible}, one would then have to assume that $R$ is compatibly F-pure along $(x_{i_1},\ldots,x_{i_t})$ for all subsets $\{x_{i_1},\ldots,x_{i_t}\}$ of $\{x_1,\ldots,x_c\}$.

\section{Comments and further questions}

The results contained in this paper hint to a strong connection between $\m$-adic stability and deformation of some F-singularities. Theorem \ref{THM P->D} shows that, in great generality, stability is a stronger property than deformation. We were not able to prove that deformation implies stability in reasonably general assumptions. However, as we are not aware of any counterexamples, it is natural to ask the following:
\begin{Question} Let $\cP$ be a property of local rings that satisfies the assumptions of Theorem \ref{THM P->D}. Are deformation and $\m$-adic stability of $\cP$ equivalent?
\end{Question}

This article deals with four important types of F-singularities. One can ask the same questions for other F-singularities, for instance:
\begin{Question} Let $(R,\m)$ be a local ring of prime characteristic $p>0$. Are F-full singularities stable? Are F-antinilpotent singularities stable?
\end{Question}
We observe that both F-fullness and F-antinipotency are known to deform \cite[Theorem 1.1]{MaQuy}. Another important F-singularity is F-nilpotency (see for instance \cite{PhamPolstra} for the defintion). We thank Pham Hung Quy for pointing out to us that F-nilpotency satisfies the conditions of Theorem \ref{THM P->D} by \cite[Theorem B]{KenkelMaddoxPolstraSimpson}, and does not deform \cite[Example 2.7]{SrinivasTakagi}. Hence it cannot be stable.

In Theorem \ref{Theorem Q-Gorenstein} we prove that if $(R,\m)$ is $\Q$-Gorenstein, then strong F-regular singularities are stable. It is therefore natural to ask the following:
\begin{Question}
If $(R,\m)$ is $\Q$-Gorenstein, is F-purity stable?
\end{Question}
Note that, with the $\Q$-Gorenstein assumption, F-purity is known to deform if the index of the canonical ideal is not divisible by the characteristic of $R$.

Another direction of investigation is whether similar stability results hold for the characteristic zero counterparts to F-singularities:
\begin{Question} Does stability of Du Bois, rational, log-canonical and log-terminal singularities hold?
\end{Question}

We conclude by mentioning that the bounds on $N$ in our theorems are constructive. Namely, the constant $N$ that we need is affected by the following considerations.
First, if $\ul{x}$ is a part of a system of parameters, then we may preserve this condition by taking $N$ to be at least the Hilbert-Samuel multiplicity of $R/(\ul{x})$ (or extendend degree of $R/(\ul{x})$, in the non-Cohen-Macaulay case). Second, controlling Hilbert--Kunz multiplicities requires to bound the $\m$-adic order of 
a discriminant. Last, for Theorem \ref{Theorem F-inj strictly filter}, we need to determine what power of $\m$ annihilates the $\m$-secondary component of $H^i_\m(R)$, for all $i \in \mathbb Z$. If $R$ is the homomorphic image of an $n$-dimensional regular local ring $S$ (e.g., in the complete case), this amounts to find what power of $\m$ annihilates the $\m$-primary component of $\Ext^{n-i}_S(R,S)$.
  
It would be good to find better explicit bounds.

\bibliographystyle{alpha}
\bibliography{References}
\end{document}